\documentclass{amsart}

\DeclareSymbolFont{AMSb}{U}{msb}{m}{n}
\DeclareSymbolFontAlphabet{\Bbb}{AMSb}
\usepackage{latexsym}
\usepackage[dvips]{graphicx}
\usepackage{amstext,amsfonts,amsmath,amsthm,graphicx,amssymb,amscd,epsfig}
\usepackage{enumerate}
\newtheorem{teo}{Theorem}[section]

\newtheorem{definition}[teo]{Definition}

\newtheorem{cor}[teo]{Corollary}
\newtheorem{prop}[teo]{Proposition}
\newtheorem{rem}[teo]{Remark}
\newtheorem{ex}[teo]{Example}

\newcommand{\norm}[1]{\left\Vert\,#1\right\Vert}
\newcommand{\abs}[1]{\left\vert\,#1\right\vert}

\newcommand{\R}{\mathbb{R}}

\def\be{\begin{equation}}
\def\ee{\end{equation}}
\def\bq{\begin{eqnarray}}
\def\eq{\end{eqnarray}}
\def\beq{\begin{eqnarray}}
\def\eeq{\end{eqnarray}}
\def\ba{\begin{array}}
\def\ea{\end{array}}
\def\bi{\begin{itemize}}
\def\ei{\end{itemize}}

\newcommand{\rank}{\textrm{rank}}

\newcommand{\su}{\; + \;}

\title[Injectivity and Almost Global Stability]{Injectivity and Almost Global Stability of Hurwitz Vector Fields }

\author{\'Alvaro Casta\~neda}
\author{V\'ictor Gu\'i\~nez}
\thanks{The first author was funded by
MathAmsud STADE. The second author
was supported by DICYT Grant 041233GM}


\begin{document}

\maketitle

\begin{abstract}
We present, in dimension $n \geq 3$, a survey of examples  to: the Jacobian conjecture, the weak Markus--Yamabe conjecture. Furthermore, we show
 and construct new examples of vector fields where the origin is almost globally asymptotically stable by using the novel concept of density functions
introduced by Rantzer.

\end{abstract}

\section{Introduction}
One of the central problems on dynamical systems is to determine
conditions under which certain points or sets are attractors for
some dynamics, that is, the orbits through
points in a neighborhood of the attractor converge to them.
In the case of continuous-time, that is, flows
associated to vector fields, an analytic condition ensuring
that an equilibrium point $x^{*}$ is a local attractor is given by the negativeness of
the real part of the eigenvalues of the Jacobian matrix at $x^{*}$. Motivated by this simple observation, in \cite{MY}, L. Markus and
H. Yamabe establish their well known global stability conjecture

\medskip

\noindent {\bf{Markus--Yamabe Conjecture (MYC):}} Let $ F: \R^n
\to \R^n$ be a $C^1-$ vector field with $ F(0) = 0 $. If for any $ x
\in \R^n$ all  the eigenvalues of $JF(x)$, the Jacobian matrix of $ F $ at $ x $,
have negative real part, then the origin is a global attractor of
the system $\dot{x} = F(x)$.

\medskip

Let us recall that the vector fields satisfying the hypothesis of {\textbf{MYC}} are called Hurwitz vector fields. It is known that the MYC is true
when $ n \leq 2 $   and false when $ n \geq 3$ (see \cite{CEGMH} for a counterexample). The proofs in the planar context, both the polynomial case (G. Meisters and C. Olech in \cite{MO}) as the $ C^1-$ case (R. Fe{\ss}ler in \cite{F}, A.A. Glutsyuk in \cite{Glu} and C. Guti\'errez in \cite{Gu}) are based on a remarkable result of C. Olech \cite{O}, where the author showed that MYC (in dimension two) is equivalent to the injectivity of the map $F.$ In $\R^n $, the problem of knowing if a Hurwitz vector field is injective is known as the Weak Markus--Yamabe conjecture.

\medskip

\noindent {\bf{Weak Markus--Yamabe Conjecture (WMYC):}} If $F:
\R^n \to \R^n$ is a  $C^1-$ Hurwitz map, then $F$ is injective.

\medskip

The {\textbf{WMYC}} is true when $n \leq 2$ and, to the best of our knowledge, it has been proved in dimension $n \geq 3$ for $C^1$ Lipschitz Hurwitz maps by A. Fernandes, C. Guti\'errez and R. Rabanal in  \cite[Corollary 4]{Fernandez}.

For $ n = 2 $, the strong injectivity theorem of Guti\'errez \cite{Gu} is the following:
``A $ C^1-$map $ f : \R^2 \to \R^2 $  is injective if $ [0,\infty) \cap \textrm{Spec}(Jf(x)) = \emptyset $, for all $ x \in \R^2 $". However, this  result fails in high dimensions as shown B. Smyth and F. Xavier \cite[Theorem 4]{Smyth}:
``There exist integers $ n > 2 $ and non--injective polynomial maps $ f : \R^n  \to \R^n $ with $ [0,\infty) \cap \textrm{Spec}(Jf(x)) = \emptyset $, for all $ x \in \R^n $".

In addition, a new tool (density functions) introduced by A. Rantzer in \cite{R}, gives sufficient conditions ensuring almost global stability of an equilibrium point for a $ C^1-$vector field in $ \mathbb{R}^n $ (\emph{i.e.}, all trajectories, except for a set of initial states with zero Lebesgue measure,  converge to the equilibrium point). Recently, this tool was used by  R. Potrie and P. Monz\'on \cite{PM} to construct a vector field $ X $ in $ \R^3 $ where the origin is almost globally stable but is not a local attractor for the differential system generated for $X$. We point out that $X$ is an almost Hurwitz  field (Hurwitz vector field except in a zero Lebesgue measure set).



This article is focused on two tasks: Firstly, we will construct (a formal description will be given later) polynomial maps $ F = \lambda I + H\colon \mathbb{R}^{3}\to \mathbb{R}^{3}$ with $JH$ nilpotent, such that the{\textbf{ WMYC}} and the {\bf{Jacobian Conjecture}}  are true, giving the inverse of $ F $ explicitly.  The results obtained are strongly related with the works \cite{C} of L.A. Campbell and \cite{ChE} of M. Chamberland and A. van den Essen. Secondly, we construct two families of  three dimensional vector fields having the Rantzer's density functions stated above. The vector fields of the first family are a generalization of the
Potrie--Monz\'on's example \cite{PM} in the sense that are almost Hurwitz and the vector field restricted to the invariant plane $ z = 0 $ is a centre. Moreover, perturbing these vector fields by $ \lambda I $, we obtain a new family a Hurwitz vector fields with the origin as global attractor which are not included in the examples of \cite[Theorem 2.7]{GC} and \cite[Theorem 2.5]{CG}.
The vector fields of the second family are not almost Hurwitz and  the the existence of a density function seems to be the only way to demonstrate the almost global stability of the origin.

The paper is organized as follows. The Section $2$ is devoted to generalize the procedure of L.A. Campbell in order to obtain new examples to {\bf{WMYC}} and the {\bf{Jacobian Conjecture}} on $\R^n.$ The Section $3,$ using density functions, shows two family of vector fields such that the origin is almost globally stable. Moreover, we construct new examples to {\bf{MYC}}.


\section{Examples to Jacobian conjecture and WMYC}

In dimension three, all existing  examples and counterexamples to MYC are maps of the form $\lambda I + H$ with $\lambda < 0$ and $JH$ nilpotent (see \cite{CEGMH},\cite{CGM}, \cite{CG}). This kind of maps, in any dimension, also are important in the study of the Jacobian conjecture since H. Bass, E. Connell and D. Wright showed
in \cite{BCW} that it suffices to solve this conjecture for such maps. Indeed, the Jacobian conjecture follows from the injectivity of these maps for all dimensions.

It is worth to emphasize that the result above triggered new questions and problems as the following one. Let $ \kappa $ be a field of characteristic zero.

\medskip

\noindent {\bf{(Homogeneous) dependence problem.}} Let $H=(H_1, \ldots, H_n) \in \kappa[x_1, \ldots, x_n]^n$ (homogeneous of degree $d \geq 1$) such that $JH$ is nilpotent and $H(0) = 0.$
Does it follow that $H_1, \ldots, H_n$ are linearly dependent over $ \kappa $?

\medskip
The homogeneous problem is true in dimension three (\cite{Bondt1}) and false in dimensions bigger than five (\cite{Bondt2}).
A counterexample for the inhomogeneous problem in dimension three was given by E. Hubbers in \cite{vE}. The map
$$H =(y-x^2, z + 2x(y-x^2), -(y-x^2)^2)$$
verifies that $JH$ is nilpotent, $\rank(JH) = 2$ and $H_1, H_2, H_3$ are linearly independent over $\kappa.$
Moreover, L.A. Campbell in \cite{C} generalizes this counterexample obtaining
$$ H = (\phi(y - x^2),z + 2 x \phi(y - x^2), -(\phi(y - x^2))^2) $$
with the same properties. Here $ \phi $  is a $ C^1-$ function of a single variable. Notice that the inverse of $ F = I + H $ can be computed explicitly as follows
$$ F^{-1} = (x - \phi(y - x^2 - z),y - z - 2 x \phi(y - x^2 - z) + (\phi(y - x^2 - z))^2,z + (\phi(y - x^2 - z))^2) \, . $$
Therefore $ F = I + H $ with $ \phi \in \R[t] $ is an example in dimension three to

\medskip

\noindent{\bf{Jacobian Conjecture on $\R^n.$}}  Every polynomial map $F: \R^n \to \R^n$ such that $\det JF \equiv 1$ is a bijective map with a polynomial inverse.

\bigskip

Let us emphasize that the previous
examples have the special form
$$H=(u(x,y,z), v(x,y,z), h(u,v)), $$
which is completely studied in \cite{ChE}, which deals with polynomial maps satisfying $ H(0) = 0 $, $ h $ has no linear part and the components of $ H $ are linearly independent over $ \kappa $.
Given $A = v_x u_z - u_x v_z$ and $B = v_y u_z - u_y v_z,$ they show that if $JH$ is nilpotent
and $\deg_z u A \neq \deg_z v B$, then there exists $T \in GL_3(k)$ such that $THT^{-1}$ takes the form
\begin{equation} \label{ache}
 (g(t),v_1 z
- (b_1 + 2 v_1 \alpha x) g(t),\alpha (g(t))^2) \, ,
\end{equation}
with $ t = y + b_1 x + v_1 \alpha x^2 $, $v_1 \alpha \neq 0 $ and $ g(t) \in k[t], g(0) = 0$ and
$\deg_t g(t) \geq 1.$

The expressions $A$ and $B$  are very useful for determine if a map
$H$ is nilpotent  \cite[Proposition 3.1]{ChE}. Also they are useful for decide if the rows of $H$ are linearly independent over $ \kappa $. In fact, we obtain the following result with $ \kappa = \R.$

\begin{prop}\label{independencia}
Let  $H=(u(x,y,z),v(x,y,z),h(u(x,y,z), v(x,y,z)))$ be a polynomial map such that the components of $H$ are linearly dependent over $\R$ and $h$ has no linear part in $u$ and $v.$ Then $A = B = 0.$
\end{prop}

\begin{proof}
Let $\alpha, \beta \in \mathbb{R}^*$ such that

\begin{displaymath}
h_u \cdot (u_x, u_y, u_z) + h_v \cdot (v_x, v_y, v_z) \equiv \alpha \cdot (u_x, u_y, u_z) + \beta \cdot (v_x, v_y, v_z)
\end{displaymath}

which is equivalent to

\begin{equation}
\label{10}
(h_u - \alpha) \cdot (u_x,u_y,u_z) + (h_v - \beta) \cdot (v_x, v_y, v_z) \equiv 0.
\end{equation}

Thus, we have the following systems of equations

\begin{equation}\label{sistemas}
\left \{
\begin{array}{rcl}
(h_u - \alpha) u_x + (h_v - \beta) v_x &= &0\\
(h_u - \alpha) u_y + (h_v - \beta) v_y &= &0\\
(h_u - \alpha) u_z + (h_v - \beta) v_z &= &0.
\end{array}
\right.
\end{equation}

By using (\ref{sistemas}), we see that

$$
\begin{array}{rcl}
u_z A - v_z B & = &u_z^2 v_x - u_z v_z (u_x + v_y) + v_z^2 u_y\\\\
& = & \frac{v_z^2}{(h_u - \alpha)^2}  \{ (h_v - \beta)^2 v_x + (h_v - \beta)(h_u - \alpha)(u_x + v_y) +(h_u - \alpha)^2 u_y\}\\\\
& = & \frac{v_z^2}{(h_u - \alpha)^2} \{(h_v - \beta)^2 v_x\\\\ & & \hspace{1.5 cm} -(h_v - \beta)(h_u - \alpha)(\frac{(h_v - \beta)}{(h_u- \alpha)} v_x - v_y) - \frac{(h_u - \alpha)^2(h_v - \beta)}{(h_u -\alpha)} v_y\}\\\\
& = & 0
\end{array}
$$
over the set of points where $ h_u \neq \alpha $. Otherwise, if $h_u = \alpha$ then in (\ref{10}) we have that $(v_x, v_y, v_z) =0$  or $h_v = \beta.$ The first case immediately implies that $A = B = 0,$ the second case is not
possible due to is a contradiction with no linearity of $h$ with respect to $u$ and $v.$

In similar way is proved that $u_z A + v_z B =0$ and the result follows.

\end{proof}

\begin{teo}
\label{11}
  Consider a polynomial map of the form $ H=(u,v,h(u, v))$ such that $ H(0) = 0 $, $ h $ has no linear part and the components of $ H $ are linearly independent over $ k $. Then if $JH$ is nilpotent
and $\deg_z u A \neq \deg_z v B$, for all $ \lambda \neq 0 $ the
polynomial map $ \lambda I + H$   is injective and has inverse polynomial.
\end{teo}

\begin{proof}

According to \cite[Corollary 4.1]{ChE}, we can suppose that the components $ H_1, H_2, H_3 $ of $H$ are
as in (\ref{ache}). Namely,

\begin{eqnarray}\label{ind}
H_1(x, y, z) & = &  g( y + b_1 x + v_1 \alpha x^2 ) \, , \\
H_2(x, y, z) & = &   v_1z - (b_1 + 2 v_1 \alpha x)g( y + b_1 x + v_1 \alpha x^2 )  \,
, \nonumber \\
H_3(x, y, z) & = & \, \alpha (g( y + b_1 x + v_1 \alpha x^2 ))^2 . \nonumber
\end{eqnarray}

If

$$
\begin{array}{rcl}
u_1 & = & \lambda x + g( y + b_1 x + v_1 \alpha x^2 )\\
u_2 & = & \lambda y +  v_1z - (b_1 + 2 v_1 \alpha x)g( y + b_1 x + v_1 \alpha x^2 )\\
u_3 & = & \lambda z +  \alpha (g( y + b_1 x + v_1 \alpha x^2 ))^2,
\end{array}
$$
it is easy to obtain

\begin{equation}\label{igualdad para inversa}
\displaystyle   \gamma u_2 +\gamma b_1 u_1+ \gamma^2 v_1 \alpha u_1^2 - \gamma^2 v_1 u_3 =  y + b_1 x + v_1 \alpha x^2,
\end{equation}

\noindent where $\gamma = {1 \over \lambda}.$ Now, put $\Phi = g( y + b_1 x + v_1 \alpha x^2) = g( \gamma u_2 +\gamma b_1 u_1+ \gamma^2 v_1 \alpha u_1^2 - \gamma^2 v_1 u_3)$ and observe that
$\lambda I + H = (\lambda x + \Phi, \lambda y + v_1 z - (b_1+2 v_1 \alpha x)\Phi, \lambda z +\alpha  \Phi^2).$ By using (\ref{igualdad para inversa}), we obtain the inverse of this map which is
$$(\gamma I + P) = (\gamma x + P_1, \gamma y + P_2, \gamma z + P_3)$$
 where

\begin{equation}\label{inversa1}
\begin{array}{rcl}
P_1(x,y,z) & = &  - \gamma g(\gamma y +\gamma b_1 x+ \gamma^2 v_1 \alpha x^2 - \gamma^2 v_1 z)\\
\\
P_2(x,y,z) & = &  - \gamma^2 v_1 z + \gamma(b_1 + 2 \gamma v_1 \alpha x )g(\gamma y +\gamma b_1 x+ \gamma^2 v_1 \alpha x^2 - \gamma^2 v_1 z)\\
\\
&  &- \gamma^2 v_1 \alpha (g(\gamma y +\gamma b_1 x+ \gamma^2 v_1 \alpha x^2 - \gamma^2 v_1 z))^2\\
\\
P_3(x,y,z) & = &   -\alpha \gamma (g(\gamma y +\gamma b_1 x+ \gamma^2 v_1 \alpha x^2 - \gamma^2 v_1 z))^2.
\end{array}
\end{equation}

\end{proof}

\begin{rem}  For $ \lambda < 0 $ (resp. $ \lambda = 1$), the polynomial map $F = \lambda I + H $ of the above Theorem is an example of the WMYC (resp. the  Jacobian conjecture) on $\R^3 $.
\end{rem}

\begin{rem}
More examples for the both conjectures in dimension $n \geq 4$ can be constructed consider the following maps of \cite[Proposition 7.1.9]{vE}:

\begin{eqnarray*}
H_1(x_1,\dots,x_n) & = &  g(x_2 - a(x_1)) \, , \\
H_i(x_1,\dots,x_n) & = &  x_{i+1}  + \frac{(-1)^i}{(i-1)!} \,
a^{(i-1)}(x_1) \,
g(x_2 - a(x_1))^{i-1} \, , \; \textrm{if} \; 2 \leq i \leq n-1 \, ,\\
H_{n}(x_1,\dots,x_n) & = &  \frac{(-1)^{n}}{(n-1)!} \, a^{(n-1)}(x_1)
\,
g(x_2- a(x_1))^{n-1}
\end{eqnarray*}
where $a(x_1) \in \R[x_1]$ with $\deg a = n-1$ and  $ g(t) \in k[t], g(0) = 0$ and
$\deg_t g(t) \geq 1.$ Following the lines of the proof of the Theorem (\ref{11}), we must consider for each fixed $n \geq 4, u_i = \lambda x_i + H_i, \, i = 1, \ldots, n$ and
$$\Phi = g(x_2 - a(x_1)) = g\Big (\frac{1}{\lambda} \, u_2 -a(\frac{1}{\lambda} u) - \frac{1}{\lambda^2} \, u_3 + (-1)^n \sum_{j=4}^{n} \frac{1}{\lambda^{j-1}} \, u_j \Big )$$
for obtain the inverse (polynomial) of $\lambda I + H.$ In fact, the inverse is $\gamma I  + P$ where

\begin{eqnarray*}
P_1(x_1, \ldots, x_n) & = &  - \gamma \Phi(x_1, \ldots, x_n)\\
\\
P_{i}(x_1, \ldots, x_n) & = &  -\gamma(x_{i+1}-  \frac{(-1)^i}{(i-1)!}  a^{(i-1)}(\gamma (x_1- \Phi(x_1, \ldots, x_n))) (\Phi(x_1, \ldots, x_n))^{i-1})  ,\\\\
P_n(x_1, \ldots, x_n) & = &   - \gamma \frac{(-1)^{n}}{(n-1)!} \, a^{(n-1)}(\gamma (x_1- \Phi(x_1, \ldots, x_n))(\Phi(x_1, \ldots, x_n))^{n-1}.
\end{eqnarray*}
with $ 2 \leq i \leq n-1.$ Therefore, for $ \lambda < 0 $ (resp. $ \lambda = 1$), the polynomial map $F = \lambda I + H $  is an example of the WMYC (resp. the  Jacobian conjecture) on $\R^n $ with $n \geq 4.$

\end{rem}

\begin{ex}
We consider $n=4,$ the inverse of the map $\lambda I + (H_1,H_2,H_3, H_4)$ is $\gamma I + (P_1,P_2, P_3, P_4)$ where

\begin{eqnarray*}
P_1(x_1, x_2, x_3, x_4) & = &  -\gamma \Phi ,\\\\
P_2(x_1, x_2, x_3, x_4) & = &  -\gamma^2 x_3 + \gamma^3 x_4 - \gamma (a_1 + \gamma 2 a_2 x_1 + \gamma^2 3 a_3 x_1^2) \Phi + \gamma^2(a_2 + 3 a_3 x_1) \Phi^2   ,\\\\
P_3(x_1, x_2, x_3, x_4) & = & -\gamma^2 x_4 + \gamma (a_2 + \gamma 3 a_3 x_1) \Phi^2 - \gamma 2 a_3 \Phi^3,\\\\
P_4(x_1, x_2, x_3, x_4) & = &  - \gamma a_3 \Phi^3,
\end{eqnarray*}
where $\Phi = \Phi(x_1, x_2, x_3, x_4) = (\gamma x_2 - \gamma x_1 - \gamma^2 x_1^2 - \gamma^3 x_1^3 - \gamma^2 x_3 + \gamma^3 x_4).$

\end{ex}

\begin{rem}
A example to WMYC and Jacobian Conjecture, in dimension 4, which does not belong to the above family of maps is $$F(x,y,z,w) = (\lambda x + y, \lambda y+ x^2 - w, \lambda z + y^2, \lambda w + 2y - z) .$$ Is easy to see
that $JH$ is nilpotent and the rows of $JH$ are linearly independent over $\R$ and has inverse
$$F^{-1}(x,y,z,w) = (\gamma (x - \phi), \phi, \gamma(z-\phi^2), \gamma(w- 2 \gamma (x-\phi)\phi+ \gamma (z-\phi^2)) ),$$
where $\phi = \gamma y - \gamma^3 x^2 + \gamma^3 z + \gamma^2 w$ and $\gamma = 1/\lambda.$ Therefore, for $ \lambda < 0 $ (resp. $ \lambda = 1$), the polynomial map $F$ is an example of the WMYC (resp. the  Jacobian conjecture) on $\R^4 .$
\end{rem}

\begin{rem} In \cite[Corollary 4.2]{ChE} it is stated that a complete study of maps $H = (u,v, h(u,v))$ with $JH$ is nilpotent will be achieved when considering the case $deg_z (u A) = deg_z (vB).$ In this context, the following result gives a partial progress
to this study since we introduce a large family of maps satisfying $deg_z (u A) = deg_z (vB),$ which through a linear of change of coordinates have the form described in (\ref{ache}).
\end{rem}


\begin{prop}
The map $ \gamma I + P $ with $ P = (P_1,P_2,P_3) $  as in (\ref{inversa1})  has the following properties:

\begin{enumerate}

\item[\textbf{(P0)}] $P(x,y,z)$ has the form $(u(x,y,z), v(x,y,z), h(u(x,y,z), v(x,y,z))).$

\item[\textbf{(P1)}] The Jacobian matrix $JP$ is nilpotent and their rows  are linearly independent over $\R.$

\item[\textbf{(P2)}] $deg_z (u A) = deg_z (vB)$.

\item[\textbf{(P3)}] Under the linear change of coordinates $(\widetilde{u},\widetilde{v},\widetilde{w}) = (x, y- \gamma v_1 z, z)$ the map $ \gamma I + P $
 is transformed into $ \gamma I + \widehat{P}$ where $\widehat{P}$ has the form $(u,v,h(u,v))$ and $deg_z (u A) \neq deg_z (vB).$

\end{enumerate}

\end{prop}

\begin{proof}
\textbf{(P0)} We have $ P = (u,v,h(u,v)) $ with $ h(u,v) = -\frac{\alpha}{\gamma} \, u^2 $.

\medskip

\textbf{(P1)} By using an algebraic manipulator we  see that $(JP)^3 =0 $. We have $ P(0) = 0 $ and $ P = (P_1,P_2,h(P_1,P_2)) $ with $ h(P_1,P_2) = -\frac{\alpha}{\gamma} \, P_1^2 $. Moreover,
if we consider $\omega(x) = b1 + 2 \gamma v_1 \alpha x,$  we have

$$
\begin{array}{rcl}
B & = &  \gamma^5 v_1 \omega(x) (g'(t))^2 - 2 \gamma^6  v_1^2 \alpha g(t) (g'(t))^2- \gamma^4 v_1 g'(t)+ \\\\
&&- \gamma^5 v_1 \omega(x) (g'(t))^2 + 2 \gamma^6  v_1^2 \alpha g(t) (g'(t))^2\\\\
&= & -\gamma^4 v_1 g'(t) \neq 0 \, ,
\end{array}
$$

and

$$A = -\gamma^4 v_1 g'(t) (\omega (x) - 2 \gamma v_1 \alpha g(t)) \neq 0,$$
and the result follows from Proposition \ref{independencia}.

\medskip

\noindent\textbf{(P2)} We have

$$u(x,y,z) \cdot A = -2 \gamma^6  v_1^2 \alpha (g(t))^2  g'(t)  + \gamma^5 v_1 \omega(x) g(t) g'(t)$$
and
$$v(x,y,z) \cdot B = \gamma^6 v_1^2 z g'(t) -\gamma^5 v_1 \omega(x) g(t) g'(t) + \gamma^6  v_1^2 \alpha (g(t))^2 g'(t).$$

Finally,  $\deg_z (u A) = \max \{3k-1, 2k-1\} = 3k-1$ and $\deg_z (v B) = \max \{k,2k-1,3k-1\} = 3k-1.$

\medskip
\noindent\textbf{(P3)} By considering the linear change of coordinate proposed, it is easy to see that  the coordinates $\widehat{P_1}, \widehat{P_2}, \widehat{P_3}$ of $\widehat{P}$ are

$$
\begin{array}{rcl}
\widehat{P_1}(x,y,z) & = &  g(\gamma y + \gamma b_1 x + \gamma^2 v_1 \alpha x^2)\\\\
\widehat{P_2}(x,y,z) & = &  - \gamma^2 v_1 z + \gamma (b_1 + 2 \gamma v_1 \alpha x) g(\gamma y + \gamma b_1 x + \gamma^2 v_1 \alpha x^2)\\\\
\widehat{P_3}(x,y,z) & = &  - \alpha \gamma (g(\gamma y + \gamma b_1 x + \gamma^2 v_1 \alpha x^2))^2.
\end{array}
$$

This coordinates $\widehat{P_1}, \widehat{P_2}, \widehat{P_3}$ are the same that (\ref{ache}) with $\gamma = 1,$ thus  $deg_z (u A) \neq deg_z (vB).$

\end{proof}

Now we consider, in dimension three,  polynomial maps of the form $ F = \lambda I  + H $ with $ H(0) = 0 $ and $ JH $ nilpotent such that the components of $ H $ are linearly dependent over $ \R $.
As it was shown in \cite[Proposition 2.1]{CG}, for such  vector field $ F = \lambda I  + H ,$ there exists $ T \in Gl_3(\R) $ such that $ T^{-1} F T = \lambda I + (P,Q,0)$, where
\begin{eqnarray} \label{normal}
P(x,y,z) & = &  -b(z) \, f(a(z) \, x + b(z) \, y) \su c(z) \quad \textrm{and} \nonumber \\
Q(x,y,z) & = &   a(z) \, f(a(z) \, x + b(z) \, y) \su d(z)
\,
\end{eqnarray}
with $ a, b, c, d \in \R[z] $ and $ f \in \R[z][t] $. Also it was pointed out in \cite[Proposition 2.3]{CG} that if $ \lambda \neq 0 $, this polynomial map
 is injective and thus satisfies the conclusion of WMYC. In fact, if we consider $\gamma = \frac{1}{\lambda},$
we can find explicitly the  inverse of $ \lambda I + (P,Q,0) $, with $ P$ and $ Q $ as in (\ref{normal}). Since $ f \in \R[z][t] $, in what follows we replace $ f(t) $ by $ f(t,z) $.

\begin{prop} For $ \lambda \neq 0 $, the inverse of $ \lambda I + (P,Q,0) $, with $ P$ and $ Q $ as in (\ref{normal}), is $ \gamma I + (R,S,0) $, where $\gamma = \frac{1}{\lambda}$ and
 $$
\begin{array}{rcl}
R & = & \gamma  b(\gamma z) \,  f\Big(\gamma \{a(\gamma z) \, x + b(\gamma z) \, y- a(\gamma z) c(\gamma z) - b(\gamma z)d(\gamma z)\},\gamma z\Big) - \gamma c(\gamma z)\\
S & = &  -\gamma a(\gamma z) \,  f\Big(\gamma \{a(\gamma z) \, x + b(\gamma z) \, y- a(\gamma z) c(\gamma z) - b(\gamma z)d(\gamma z)\},\gamma z\Big) - \gamma d(\gamma z).
\,
\end{array}
$$
\end{prop}
\begin{proof} Putting
$$
\begin{array}{rcl}
u &=&\lambda x -b(z) \, f(a(z) \, x + b(z) \, y,z) \su c(z)\\\\
v &=& \lambda y +  a(z) \, f(a(z) \, x + b(z) \, y,z) \su d(z)\\\\
w &=& \lambda z \, ,
\end{array}
$$
we obtain
\begin{equation}\label{inversa}
a(z)u+b(z)v = \lambda(a(z)x+b(z)y) + a(z)c(z) + b(z) d(z) \, ,
\end{equation}
which implies
\begin{displaymath}
a(z)x+b(z)y = \gamma \big( a(z)u+b(z) v - a(z)c(z) - b(z)d(z) \big).
\end{displaymath}

Therefore, by consider that $z = \gamma w,$ we can deduce

$$
\begin{array}{lcl}
 u & = & \lambda x - b(\gamma w) \, f(m(u,v,w),\gamma w) \su c(\gamma w)\\\\
v & = &  \lambda y +  a(\gamma w) \, f(m(u,v,w),\gamma w)\su d(\gamma w),
\end{array}
$$
with $ \; m(u,v,w) = \gamma \big\{ a(\gamma w)u+b(\gamma w) v - a(\gamma w)c(\gamma w) - b(\gamma w)d(\gamma w) \big\} $.
Finally,

$$
\begin{array}{lcl}
 x & = &  \gamma (  u +b(\gamma w) \, f(m(u,v,w),\gamma w)  - c(\gamma w) ) \\\\
 y & = & \gamma ( v -a(\gamma w) \, f(m(u,v,w),\gamma w) - d(\gamma w) ).\\\\
\end{array}
$$

\end{proof}

\begin{rem}
Let $ F = \lambda I + H : \R^3 \to \R^3 $  be a polynomial map verifying $ H(0) = 0 $, $ JH $ nilpotent and their components are linearly dependent over $ \R $. If
$ \lambda < 0 $ (resp. $ \lambda = 1 $), then the map $F$ is an example to WMYC (resp. the Jacobian conjecture).
\end{rem}

\section{Examples of almost global stability}

In \cite[Theorem 3.2,Theorem 3.5]{CG} we prove that the vector fields in $ \R^3 $ of the form $ F = \lambda I  + H $, with $ \lambda < 0 $, $ H $ as
in (\ref{ind}) and $ g(t) = A_1 \, t + A_2 \, t^2 $, are counterexamples to {\textbf{MYC}} since they have unbounded orbits. Moreover, by following the
respective proofs, we can deduce the existence of an open set of initial states whose trajectories do not converge to the origin.
Thus the origin for these vector fields
is not almost global attractor in the following sense:

\begin{definition} \label{3.1}
Consider the differential equation
\begin{equation} \label{sistema}
\dot{x} = F(x)
\end{equation}
 where $F:\R^n \to \R^n$ is a $ C^1-$map and $ F(0) = 0.$ We say the origin is an almost global attractor if all trajectories, except for a set of initial states with zero Lebesgue measure,
 converge to the origin.
\end{definition}

From Definition \ref{3.1}, it arises the following question:

\medskip

\noindent {\bf Question 1.} Do there exist counterexamples to the {\textbf{MYC}} with the origin almost global attractor ?

\medskip

In \cite{R}, A. Rantzer introduces a new tool, namely, the density functions in order to obtain  sufficient conditions for almost global stability of an equilibrium point for a $ C^1-$vector field in $ \mathbb{R}^n $.
\begin{definition}
\label{density}
A density function of \textnormal{(\ref{sistema})} is a $C^{1}$ map
$\rho\colon \mathbb{R}^{n}\setminus \{0\}\to [0,+\infty)$, integrable outside a ball
centered at the origin that satisfies
\begin{displaymath}
[\triangledown \cdot \rho F](x) >0
\end{displaymath}
almost everywhere with respect to $\mathbb{R}^{n},$ where
$$
\triangledown \cdot [\rho F] = \triangledown \rho \cdot F + \rho[\triangledown\cdot F],
$$
and $\triangledown \rho$, $\triangledown \cdot F$ denote respectively the gradient of $\rho$ and the divergence
of $F$.
\end{definition}
The main result of A. Rantzer \cite{R} is the following.
\begin{teo}
\label{Rantzer}
Given the differential system
$$\dot{x} = F(x),$$
where $F \in C^1, \, F(0) = 0$, suppose there exists a density function $\rho : \R^n \setminus \{0\} \to [0, +\infty)$ such
that $ \rho (x) F(x) / \norm{x}$   is integrable on $\{x \in \R^n: \norm{x} \geq 1\}.$
Then, almost all trajectories converge to the origin, i.e., the
origin is almost globally stable.
\end{teo}

In this context, the next question arises naturally:

\noindent {\bf Question 2.} Do there exist counterexamples  to the {\textbf{MYC}} that support density functions ?

In this paper the best answer we give is the following family of almost Hurwitz vector fields; i.e. of vector fields for which the Hurwitz condition hold over  $ \R^n - A $ with  $ A $ a zero Lebesgue measure set.
\begin{teo}\label{123}
Consider the real numbers  $  c \leq a < 0, \, b \in \R, \, k\geq 1,$ and the polynomial $R(z) = \sum\limits_{i=1}^k a_{2i} z^{2i}$ with $a_{2i}>0$ for  $i=1, \ldots, k.$  Then
\begin{equation}
\label{casihurwitz}
 F(x,y,z) = (y,-x,0) +((ax + by)R(z),(-bx + cy)R(z),-z R(z))
\end{equation} is an almost Hurwitz vector field. Moreover, this vector field  has associated the density function $\rho(x,y,z) = (x^2+y^2+R(z))^{-\alpha}$ with $ \; \alpha > \max\{2,
\frac{a+c-1-2k}{2a}, \frac{3-a-c}{2}\} $.
\end{teo}

\begin{proof}
The Jacobian matrix of $F$ in a point $ (x,y,z) $ is
$$
JF(x,y,z) = \left (\begin{array}{ccc}
a R(z) & 1 +  b R(z)& *\\
-(1+b R(z)) &c  R(z)&  *\\
0 & 0 & - (R(z) + zR'(z))
\end{array}
\right ).
$$

Then $ \; \lambda_3 = - (R(z) + zR'(z)) $ is an eigenvalue. The others two eigenvalues $ \lambda_1, \lambda_2 $ verify
$$ \lambda_1  + \lambda_2 = (a + c) \; R(z) \quad \textrm{and} \quad \lambda_1 \lambda_2 = ac R(z)^2 + (1 + b R(z))^2 \, . $$
Therefore, for $ z \neq 0 $ (resp. $ z = 0 $),  we have $ \lambda_3 < 0 $ (resp. $ \lambda_3 = 0 $) and $ \lambda_1 $ and  $ \lambda_2 $ have negative real part (resp. null real part).
In addition, $F$ verifies the Hurwitz condition except in the invariant plane $z=0.$

In what follows we prove that
the map $$\rho(x,y,z) = (x^2 + y^2+ R(z))^{-\alpha} \, , $$
under the conditions of the theorem is a density function for the vector field $ F $.
The condition  $\alpha > 2$ ensures the integrability  of $ \rho(x,y,z)  $ outside the ball centered at the origin of radius one.

It remains to prove that $\nabla \cdot (\rho F)(x,y,z)$ is positive almost everywhere in $\R^3.$ We have
$$ \triangledown \rho (x,y,z) = \frac{-\alpha}{(x^2+y^2+R(z))^{\alpha + 1}} \; (2 x, 2 y, R'(z))  \, , $$
and
$$ [\triangledown\cdot F] (x,y,z) = (a + c -1) R(z) - z R'(z) \, . $$
Then
\begin{eqnarray*}
[\nabla \cdot \rho F](x,y,z) & = & (\triangledown \rho \cdot F) (x,y,z) + \rho(x,y,z) \; [\triangledown\cdot F] (x,y,z) \\
& = & \frac{-\alpha \, R(z)}{(x^2+y^2+R(z))^{\alpha +1}} \, [2(a x^2 + c y^2) - z \, R'(z)] \\
&  & +\frac{1}{(x^2+y^2+R(z))^{\alpha}} \, [(a + c -1) \, R(z) - z \, R'(z)] \\
& = & \frac{1}{(x^2+y^2+R(z))^{\alpha + 1}} \, \left[-2 \alpha (a x^2 + c y^2) R(z) + \alpha z R(z) R'(z) \right. \\
&  & + \left. (x^2+y^2+R(z)) \, [(a + c -1) \, R(z) - z \, R'(z)]\right] \\
& = & \frac{1}{(x^2+y^2+R(z))^{\alpha + 1}} \, \left[((a + c - 1) R(z) + (\alpha - 1) z R'(z)) R(z) \right. \\
&  & + ((a + c - 1 - 2 \alpha a) R(z) - z R'(z)) x^2  \\
&  & +  \left.((a + c - 1 - 2 \alpha c) R(z) - z R'(z)) y^2 \right] \, .
\end{eqnarray*}
Since $ R(z) = \sum\limits_{i=1}^k a_{2i} z^{2i}$  and $ z R'(z) = \sum\limits_{i=1}^k 2 i a_{2i} z^{2i}$ with $a_{2i}>0$ for  $i=1, \ldots, k $, we obtain $ [\triangledown\cdot F] (x,y,z) > 0 $ for $ z \neq 0 $ and $ [\triangledown\cdot F] (x,y,z) = 0 $ for $ z = 0 $, if
$$ a + c - 1 + 2 (\alpha - 1) i > 0 \, , \quad   a + c - 1 - 2 \alpha a - 2 i > 0 \, , \quad   a + c - 1 - 2 \alpha c - 2 i > 0 \, , $$
for $i=1, \ldots,k $. Then, the proof is finished due to these inequalities are consequence of our hypothesis $ \; \alpha > \max\{2,
\frac{a+c-1-2k}{2a}, \frac{3-a-c}{2}\} $.
\end{proof}

This family is a generalization of the following example of R. Potrie and P. Monz\'on (\cite{PM})

\begin{equation}\label{PM}
F(x,y,z) = (y -2xz^2,-x-2yz^2, -z^3 ) \, ,
\end{equation}
which has density function $ \rho(x,y,z) = (x^2 + y^2 + z^2)^{-4} $.

\begin{cor} \label{1234} The vector field $ F $ given by (\ref{casihurwitz}) under the conditions of Theorem \ref{123} has the origin as an almost global attractor which is not locally asymptotically stable.
\end{cor}
\begin{proof}  We have $ F(x,y,0) = (y,-x,0) $, then the origin is not locally asymptotically stable. On the other hand, to prove that the origin is almost global attractor we use  Rantzer's result (Theorem \ref{Rantzer}). Then it is sufficient to show that the  condition  $\alpha > 2$ ensures the integrability  of $ \rho(x,y,z) F(x,y,z)/ \norm{(x,y,z)} $ outside the ball centered at the origin of radius one. In fact, if we consider
$k_0 = \max \{1,a_1\} $,  $ k_1 = \max\{1, b^2+c^2 + \abs{b} (a-c)\}$ and $ k_2 = 2 \abs{b} +  a-c $, we have  that

\begin{eqnarray*}
\norm{F(x,y,z)}^2 & = & x^2 + y^2 + R(z)^2 \, \left[(a^2 + b^2) x^2 + (b^2 + c^2) y^2 + z^2 + 2 b (a - c)xy\right] \\
&  & + 2 R(z) \, \left[(a - c) xy + b (x^2 + y^2)\right] \\
& \leq &  x^2 + y^2 + R(z)^2 \, \left[(b^2 + c^2) (x^2 + y^2) + z^2 + \abs{b} (a - c) (x^2 + y^2)\right]  \\
&  & + R(z) (a - c + 2\abs{b}) (x^2 + y^2) \\
& \leq  &  (x^2+y^2+z^2) + k_1 (x^2+y^2+R(z))^2 (x^2+y^2+z^2)  \\
& & + \; k_2 (x^2+y^2+R(z))(x^2+y^2+z^2).
\end{eqnarray*}
and that $x^2 + y^2 + R(z) \geq k_0.$
Thus, this facts combined with the assumption over $\alpha$ imply that

\begin{eqnarray*}
\frac{\norm{F(x,y,z)}^2 \rho(x,y,z)^2}{\norm{(x,y,z)}^2} & \leq &  \frac{1}{(x^2+y^2+R(z))^{2 \alpha}} + \frac{k_1}{(x^2+y^2+R(z))^{2 \alpha-2}}  \\
& & + \frac{k_2}{(x^2+y^2+R(z))^{2 \alpha-1}} \, \cdot
\end{eqnarray*}
Therefore $ \rho(x,y,z) F(x,y,z)/ \norm{(x,y,z)} $ is integrable outside the ball centered at the origin of radius one.
\end{proof}

\begin{rem}
 An alternative --and very simple-- proof of Corollary \ref{1234} is obtained by considering
the fact that $\langle F(x,y,z), (x,y,z) \rangle = (a x^2 + b y^2 - z^2) R(z) < 0 $ for all $ z \neq 0 $ combined with the invariance of the plane $z=0$.
\end{rem}
 Motivated by the last remark, we have the following result.

 \begin{prop} \label{prop1} Let $ F : \R^3 \to \R^3 $ be a $ C^1-$vector field with $ F(0) = 0 $ such that the plane $ z = 0 $ is invariant and $ [\triangledown\cdot F] (x,y,z) < 0 $ for all $ z \neq 0 $. Suppose that there exists a positive $ C^1-$function $ \rho : \R^3-\{0\} \to \R $ satisfying:
 \begin{enumerate}
 \item[a)] $ [\triangledown\cdot \rho F] (x,y,z) > 0 $ for all $ z \neq 0 $,
 \item[b)] $ \lim_{p \to 0} \, \rho(p) = \infty $.
 \item[c)] $ \lim_{\|p\| \to \infty} \, \rho(p) = 0 $.
 \end{enumerate}
 Then, for any initial state $ \alpha(0) = (x(0),y(0),z(0)) $ verifying $ z(0) \neq 0 $,  the trajectory $ \alpha(t) $ exists for all $ t \in [0,\infty[ $ and tends to zero as $ t \to \infty $.
\end{prop}
\begin{proof} Condition a) implies $ [\triangledown\cdot \rho F] (x,y,z) \geq 0 $ everywhere, and
$ (\triangledown \rho \cdot F) (x,y,z) > -\rho(x,y,z) \; [\triangledown\cdot F] (x,y,z) > 0 $ for all $ (x,y,z) $ with $ z \neq 0 $.

\noindent Let us consider the function $ V : \R^3 \to \R $ defined by $ V = \rho^{-1} $ outside the origin and $ V(0) = 0 $. Then $ V $ is continuous by b), it is $C^1 $ outside the origin and
\begin{enumerate}
\item[1)] $ V(p) \geq 0 $ for all $ p $ and $ V(p) = 0 $ if and only if $ p = 0 $.
\item[2)] $ (\triangledown V \cdot F) (x,y,z) = -\rho(x,y,z)^{-2} \, (\triangledown \rho \cdot F) (x,y,z) < 0 $, for $ z \neq 0 $.
\item[3)] $  \lim_{\|p\| \to \infty} \, V(p) = \infty $.
\end{enumerate}
Then, if $ \alpha(0)  = (x(0),y(0),z(0)) $ is a initial state with $ z(0) \neq 0 $, we have $ \frac{d}{dt} V(\alpha(t)) < 0 $ by 2). Then the trajectory $ \alpha(t) $ remains over the sublevel $ V(\alpha(0)) $, which is bounded by 3). Therefore, $ \alpha(t) $ exists for all $ t \in [0,\infty[ $ and tend to 0 as $ t \to \infty $.
\end{proof}

\begin{cor} \label{cor1} If in Proposition \ref{prop1} we have $ [\triangledown\cdot \rho F] (x,y,z) > 0 $ for all $ z \in \R $, then the origin is globally asymptotically stable.
\end{cor}

If we perturb the vector field $ F $ defined by (\ref{casihurwitz}) with $ \lambda I $ and $ \lambda < 0 $, then  $ G = \lambda I + F $ is a Hurwitz vector field, which is another example of {\bf MYC}, by the following
\begin{prop} \label{prop2} Let $ G = \lambda I + F $ with $ \lambda < 0 $ and $ F $ as in (\ref{casihurwitz}) that verify the conditions of Theorem \ref{123}. Then $ G $ is Hurwitz, admits the same density  function $ \rho(x,y,z) = (x^2 + y^2 + R(z))^{\alpha} $ but with $ \alpha > 3 $, and the origin is globally asymptotically stable.
\end{prop}
\begin{proof} $ G $ is Hurwitz  due to the fact that $ F $ is almost Hurwitz and $ JG = \lambda I  + JF $. Since
$$ [\triangledown\cdot \rho G] (x,y,z) = [\triangledown\cdot \rho F] (x,y,z) + \lambda \frac{(3 - 2\alpha)(x^2 + y^2) + 3 R(z) - \alpha z R'(z)}{(x^2 + y^2 + R(z))^{\alpha + 1}} \; ,$$
under the additional condition $ \alpha > 3 $, $ \rho $ is a density function for $ G $. Finally, the global asymptotic stability of the origin follows from Corollary  \ref{cor1}.
\end{proof}

The previous results imply that if we are interested  in finding vector fields $F$ with the plane $ z = 0 $ invariant supporting density
functions (but without obvious Lyapunov functions), these fields must verify
$ \triangledown\cdot F >0$ over some open set. The following family of vector fields satisfy these conditions.

\begin{prop} \label{prop3.10}
Consider the vector field $ F = (P,Q,R) $ defined by
\begin{eqnarray*}
P(x,y,z) & = & -x + A_2 \, y  + a_1 \, x^3 + 3 \, a_2 \, x^2 \, y + 3 \, a_3 \, x \, y^2 + a_4 \, y^3 + 3 \, a_5 \, x^2 \, z + \\
 & & 6 \, a_6 \, x \, y \, z + 3 \, a_7 \, y^2 \, z + 3 \, a_8 \, x \, z^2 + 3 \, a_9 \, y \, z^2 + a_{10} \, z^3 \, , \\
Q(x,y,z) & = & -A_2 \, x - \, y  + b_1 \, x^3 + 3 \, b_2 \, x^2 \, y + 3 \, b_3 \, x \, y^2 + b_4 \, y^3 + 3 \, b_5 \, x^2 \, z + \\
& & 6 \, b_6 \, x \, y \, z + 3 \, b_7 \, y^2 \, z + 3 \, b_8 \, x \, z^2 + 3 \, b_9 \, y \, z^2 + b_{10} \, z^3  \, , \\
R(x,y,z) & = &  z \, (-1  +  3 \, c_5 \, x^2  +
6 \, c_6 \, x \, y  + 3 \, c_7 \, y^2 + 3 \, c_8 \, x \, z + 3 \, c_9 \, y \, z + c_{10} \, z^2) \, ,
 \end{eqnarray*}
 with

\begin{eqnarray*}
a_8 & = & \frac{2}{3} - a_3 \, , \; a_{10} = \frac{3}{2} \, (a_5 - a_7) \, , \\
b_1 & = & -3 \, a_2 + a_4 + 3 \, b_3 \, , \; b_2  =  1 - a_3 - \frac{c_{10}}{3} \, , \; b_4 =  3 - a_1 - c_{10} \, ,  \\
b_6  & = & \frac{1}{2} \, (a_5 - a_7) \, , \; b_7 = 2 \, a_6  +  b_5 \, , \; b_8  = 2 \, a_2 - a_4 -  a_9 - b_3  \, , \\
b_9 & = & -\frac{1}{3} + a_3 + \frac{c_{10}}{3} \, , \;  b_{10}  =  3 \, a_6 \, , \\
c_5 & = & -\frac{2}{3} + \frac{a_1}{3} + a_3 + \frac{c_{10}}{3} \, , \; c_6  =  \frac{2 a_4}{3} -  a_2 + b_3 \, , \; c_7  =  \frac{4}{3} - \frac{a_1}{3} - a_3  - \frac{c_{10}}{3} \, ,\\
c_8 & = & \frac{1}{2} \, (a_5 + a_7) \, , \; c_9 = a_6 + b_5 \, .
\end{eqnarray*}
Then $ F $ has the density function
$$ \rho(x,y,z) \; = \; \frac{1}{(x^2 + y^2 + z^2)^2} \; \cdot  $$
Moreover the plane $ z= 0 $ is invariant and the coefficients of $ F $ can be chosen such that $ \triangledown\cdot F $ is positive over some open set.
\end{prop}
\begin{proof}  Clearly the plane $ z = 0 $ is invariant. After  straightforward computations we obtain
$$  [\triangledown\cdot \rho F](x,y,z) \; = \; \frac{1 + x^2 + y^2 + z^2}{(x^2 + y^2 + z^2)^2} \, ,$$
and that the integrability condition of $ \rho $ holds.
Moreover
 \begin{eqnarray*}
 \triangledown\cdot F \, (x,y,z) & = & -3 + (1 + 4 a_1) x^2  + (13 - 4 a_1 - 4 c_{10}) y^2 + (24 a_6 + 12 b_5) y z \\
 &  & + \, (1 + 4 c_{10}) z^2 +
 x ((4 a_4 + 12 b_3) y + 12 a_5 z) \, ,
 \end{eqnarray*}
assume positive values. For example, if  $ y = 0 $, we have $ \; \triangledown\cdot F \, (x,0,z) = -3 +  (1 + 4 a_1) x^2 + (1 + 4 c_{10}) z^2 + 12 a_5 x z $ which is positive for $ x, z $ sufficient large and  $ a_1 > 0 $ and $ a_5 < 0 $.
\end{proof}

\begin{rem} The coefficients of the vector field $ F = (P,Q,R) $ of Proposition \ref{prop3.10} were obtained from the equation
$$   [\triangledown\cdot \rho F](x,y,z) \; = \; \frac{1 + x^2 + y^2 + z^2}{(x^2 + y^2 + z^2)^2} \, ,$$
which is equivalent to
$$ -4 \, (x P + y Q + z R) \; + \; (x^2 + y^2 + z^2) \, [\triangledown\cdot F] \; = \; (1 + x^2 + y^2 + z^2) \, (x^2 + y^2 + z^2) \, . $$
\end{rem}



\end{document}